\documentclass{amsart}

\newtheorem{thm}{Theorem}
\newtheorem{lem}[thm]{Lemma}
\newtheorem{cor}[thm]{Corollary}

\theoremstyle{definition} \theoremstyle{example}
\theoremstyle{remark} \numberwithin{equation}{section}

\begin{document}

\begin{center}
{\bf Infinite norm decompositions of C$^*$-algebras}
\end{center}

\begin{center}
{\it Farkhad Arzikulov}
\end{center}

\email{arzikulovfn@yahoo.com}

\curraddr{Institute of Mathematics and Information technologies,
Tashkent}

 subjclass: {\it Primary} 46L35, 17C65; {\it Secondary} 47L30

keywords: {\it C$^*$-algebra, infinite norm decompositions}

\begin{center}
{\bf Abstract}
\end{center}

In the given article the notion of infinite norm decomposition of
a C$^*$-algebra is investigated. The norm decomposition is some
generalization of Peirce decomposition. It is proved that the
infinite norm decomposition of any C$^*$-algebra is a
C$^*$-algebra. C$^*$-factors with an infinite and a nonzero finite
projection and simple purely infinite C$^*$-algebras are
constructed.

\section*{Introduction}

In the given article the notion of infinite norm decomposition of
C$^*$-algebra is investigated. It is known that for any projection
$p$ of a unital C$^*$-algebra $A$ the next equality is valid
$A=pAp\oplus pA(1-p)\oplus (1-p)Ap\oplus (1-p)A(1-p)$, where
$\oplus$ is a direct sum of spaces. The norm decomposition is some
generalization of Peirce decomposition. First such infinite
decompositions were introduced in [1] by the author.

In this article a unital C$^*$-algebra $A$ with an infinite
orthogonal set $\{p_i\}$ of equivalent projections such that
$\sup_i p_i=1$, and the set $\sum_{ij}^o
p_iAp_j=\{\{a_{ij}\}:\,\text{for any indexes}\,i,j,\, a_{ij}\in
p_i Ap_j,\,\text{and}\, \Vert
\sum_{k=1,\dots,i-1}(a_{ki}+a_{ik})+a_{ii}\Vert\to
0\,\text{at}\,\, i\to \infty\}$ are considered. Note that all
infinite sets like $\{p_i\}$ are supposed to be countable.

The main results of the given article are the next:

- for any C$^*$-algebra $A$ with an infinite orthogonal set
$\{p_i\}$ of equivalent projections such that $\sup_i p_i=1$ the
set $\sum_{ij}^o p_iAp_j$ is a C$^*$-algebra with the
componentwise algebraic operations, the associative multiplication
and the norm.

- there exist a C$^*$-algebra $A$ and different countable
orthogonal sets $\{e_i\}$ and $\{f_i\}$ of equivalent projections
in $A$ such that $\sup_i e_i=1$, $\sup_i f_i=1$, $\sum_{ij}^o
e_iAe_j\neq \sum_{ij}^o f_iAf_j$.

- if $A$ is a W$^*$-factor of type II$_\infty$, then there exists
a countable orthogonal set $\{p_i\}$ of equivalent projections in
$A$ such that $\sum_{ij}^o p_iAp_j$ is a C$^*$-factor with a
nonzero finite and an infinite projection. In this case
$\sum_{ij}^o p_iAp_j$ is not a von Neumann algebra.

- if $A$ is a W$^*$-factor of type III, then for any countable
orthogonal set $\{p_i\}$ of equivalent projections in $A$. The
C$^*$-subalgebra $\sum_{ij}^o p_iAp_j$ is simple and purely
infinite. In this case $\sum_{ij}^o p_iAp_j$ is not a von Neumann
algebra.

- there exist a C$^*$-algebra $A$ with an orthogonal set $\{p_i\}$
of equivalent projections such that $\sum_{ij}^o p_iAp_j$ is not
an ideal of $A$.

\newpage

\section{infinite norm decompositions}

\begin{lem}
Let $A$ be a C$^*$-algebra, $\{p_i\}$ be an infinite orthogonal
set of projections with the least upper bound $1$ in the algebra
$A$ and let $\mathcal{A}=\{\{p_iap_j\}: a\in A\}$. Then,

1) the set $\mathcal{A}$ is a vector space with the next
componentwise algebraic operations
$$
\lambda \{p_iap_j\}=\{p_i \lambda a p_j\}, \lambda\in {\Bbb C}
$$
$$
\{p_iap_j\}+\{p_ibp_j\}=\{p_i(a+b)p_j\}, a,b\in A,
$$

2) the algebta $A$ and the vector space $\mathcal{A}$ can be
identified in the sense of the next map
$$
\mathcal{I}: a\in A\to \{p_iap_j\}\in \mathcal{A}.
$$
\end{lem}

\begin{proof}
Item 1) of the lemma can be easily proved.

Proof of item 2): We assert that $\mathcal{I}$ is a one-to-one
map. Indeed, it is clear, that for any $a\in A$ there exists a
unique set $\{p_iap_j\}$, defined by the element $a$.

Suppose that there exist different elements $a$ and $b$ in $A$
such that $p_iap_j=p_ibp_j$ for all $i$, $j$, i.e.
$\mathcal{I}(a)=\mathcal{I}(b)$. Then $p_i(a-b)p_j=0$ for all $i$
and $j$. Observe that
$p_i((a-b)p_j(a-b)^*)=((a-b)p_j(a-b)^*)p_i=0$ and
$(a-b)p_j(a-b)^*\geq 0$ for all  $i$, $j$. Therefore, the element
$(a-b)p_j(a-b)^*$ commutes with every projection in $\{p_i\}$.

We prove $(a-b)p_j(a-b)^*=0$. Indeed, there exists a maximal
commutative subalgebra $A_o$ of the algebra $A$, containing the
set $\{p_i\}$ and the element $(a-b)p_j(a-b)^*$. Since
$(a-b)p_j(a-b)^*p_i=p_i(a-b)p_j(a-b)^*=0$ for any $i$, then the
condition $(a-b)p_j(a-b)^*\neq 0$ contradicts the equality $\sup_i
p_i=1$.

Indeed, in this case $p_i\leq 1-1/\Vert (a-b)p_j(a-b)^*\Vert
(a-b)p_j(a-b)^*$ for any $i$. Since by $(a-b)p_j(a-b)^*\neq 0$ we
have $1>1-1/\Vert (a-b)p_j(a-b)^*\Vert (a-b)p_j(a-b)^*$, then we
get a contradiction with $\sup_i p_i=1$. Therefore
$(a-b)p_j(a-b)^*=0$.

Hence, since $A$ is a C$^*$-algebra, than
$\Vert(a-b)p_j(a-b)^*\Vert=\Vert((a-b)p_j)((a-b)p_j)^*\Vert=
\Vert((a-b)p_j)\Vert\Vert((a-b)p_j)^*\Vert=\Vert(a-b)p_j\Vert^2=0$
for any $j$. Therefore $(a-b)p_j=0$, $p_j(a-b)^*=0$ for any $j$.
Analogously, we can get $p_j(a-b)=0$, $(a-b)^*p_j=0$ for any $j$.
Hence the elements $a-b$, $(a-b)^*$ commute with every projection
in $\{p_i\}$. Then there exists a maximal commutative subalgebra
$A_o$ of the algebra $A$, containing the set $\{p_i\}$ and the
element $(a-b)(a-b)^*$. Since $p_i(a-b)(a-b)^*=(a-b)(a-b)^*p_i=0$
for any $i$, then the condition $(a-b)(a-b)^*\neq 0$ contradicts
the equality $\sup_i p_i=1$.

Therefore, $(a-b)(a-b)^*=0$, $a-b=0$, i.e. $a=b$. Thus the map
$\mathcal{I}$ is one-to-one.
\end{proof}

\begin{lem}
Let $A$ be a C$^*$-algebra, $\{p_i\}$ be an infinite orthogonal
set of projections with the least upper bound $1$ in the algebra
$A$ and $a\in A$. Then, if $p_iap_j=0$ for all $i$, $j$, then
$a=0$.
\end{lem}

\begin{proof}
Let $p\in \{p_i\}$. Observe that
$p_iap_ja^*=p_i(ap_ja^*)=ap_ja^*p_i=(ap_ja^*)p_i=0$ for all $i$,
$j$ and $ap_ja^*=ap_jp_ja^*= (ap_j)(p_ja^*=(ap_j)(ap_j)^*\geq 0$.
Therefore, the element $ap_ja^*$ commutes with all projections of
the set $\{p_i\}$.

We prove $ap_ja^*=0$. Indeed, there exists a maximal commutative
subalgebra $A_o$ of the algebra $A$, containing the set $\{p_i\}$
and the element $ap_ja^*$. Since $p_i(ap_ja^*)=(ap_ja^*)p_i=0$ for
any $i$, then the condition $ap_ja^*\neq 0$ contradicts the
equality $\sup_i p_i=1$ (see the proof of lemma 1). Hence
$ap_ja^*=0$.

Hence, since $A$ is a C$^*$-algebra, then $\Vert
ap_ja^*\Vert=\Vert(ap_j)(ap_j)^*\Vert=
\Vert(ap_j)\Vert\Vert(ap_j)^*\Vert=\Vert ap_j\Vert^2=0$ for any
$j$. Therefore $ap_j=0$, $p_ja^*=0$ for any $j$. Analogously we
have $p_ja=0$, $a^*p_j=0$ for any $j$. Hence the elements $a$,
$a^*$ commute with all projections of the set $\{p_i\}$. Then
there exists a maximal commutative subalgebra $A_o$ of the algebra
$A$, containing the set $\{p_i\}$ and the element $aa^*$. Since
$p_iaa^*=aa^*p_j=0$ for any $i$, then the condition $aa^*\neq 0$
contradicts the equality $\sup_i p_i=1$ (see the proof of lemma
1). Hence $aa^*=0$ and $a=0$.
\end{proof}

\begin{lem}
Let $A$ be a C$^*$-algebra of bounded linear operators on a
Hilbert space $H$, $\{p_i\}$ be an infinite orthogonal set of
projections in $A$ with the least upper bound $1$ in the algebra
$B(H)$ è $a\in A$. Then $a\geq 0$ if and only if for any finite
subset $\{p_k\}_{k=1}^n\subset \{p_i\}$ the inequality $pap\geq 0$
holds, where $p=\sum_{k=1}^n p_k$.
\end{lem}

\begin{proof}
By positivity of the operator $T: a\to bab, a\in A$ for any $b\in
A$, if $a\geq 0$, then for any finite subset
$\{p_k\}_{k=1}^n\subset \{p_i\}$ the inequality $pap\geq 0$ holds.

Conversely, let $a\in A$. Suppose that for any finite subset
$\{p_k\}_{k=1}^n\subset \{p_i\}$ the inequality $pap\geq 0$ holds,
where $p=\sum_{k=1}^n p_k$.

Let $a=c+id$, for some nonzero self-adjoint elements $c$, $d$ in
$A$. Then
$(p_i+p_j)(c+id)(p_i+p_j)=(p_i+p_j)c(p_i+p_j)+i(p_i+p_j)d(p_i+p_j)\geq
0$ for all $i$, $j$. In this case the elements
$(p_i+p_j)c(p_i+p_j)$ and $(p_i+p_j)d(p_i+p_j)$ are self-adjoint.
Then $(p_i+p_j)d(p_i+p_j)=0$ and $p_idp_j=0$ for all $i$, $j$.
Hence by lemma 2 we have $d=0$. Therefore $a=c=c^*=a^*$, i.e.
$a\in A_{sa}$. Hence, $a$ is a nonzero self-adjoint element in
$A$. Let $b_n^\alpha=\sum_{kl=1}^n p_k^\alpha ap_l^\alpha$ for all
natural numbers $n$ and finite subsets
$\{p_k^\alpha\}_{k=1}^n\subset \{p_i\}$. Then the set
$(b_n^\alpha)$ ultraweakly converges to the element $a$.

Indeed, we have $A\subseteq B(H)$. Let $\{q_\xi\}$ be a maximal
orthogonal set of minimal projections of the algebra $B(H)$ such,
that $p_i=\sup_\eta q_\eta$, for some subset $\{q_\eta\}\subset
\{q_\xi\}$, for any $i$. For arbitrary projections $q$ and $p$ in
$\{q_\xi\}$ there exists a number $\lambda \in {\Bbb C}$ such,
that $qap=\lambda u$, where $u$ is an isometry in $B(H)$,
satisfying the conditions $q=uu^*$, $p=u^*u$. Let
$q_{\xi\xi}=q_\xi$, $q_{\xi\eta}$ be such element that
$q_\xi=q_{\xi\eta}q_{\xi\eta}^*$,
$q_\eta=q_{\xi\eta}^*q_{\xi\eta}$ for all different $\xi$ and
$\eta$. Then, let $\{\lambda_{\xi\eta}\}$ be a set of numbers such
that $q_\xi aq_\eta=\lambda_{\xi\eta}q_{\xi\eta}$ for all $\xi$,
$\eta$. In this case, since $q_\xi aa^*q_\xi=q_\xi(\sum_\eta
\lambda_{\xi\eta}\bar{\lambda}_{\xi\eta})q_\xi<\infty$, the
quantity of nonzero numbers of the set
$\{\lambda_{\xi\eta}\}_\eta$ ($\xi$-th string of the infinite
dimensional matrix $\{\lambda_{\xi\eta}\}_{\xi\eta}$) is not
greater then the countable cardinal number and the sequence
$(\lambda_n^\xi)$ of all these nonzero numbers converges to zero.
Let $v_{q_\xi}$ be a vector of the Hilbert space $H$, which
generates the minimal projection $q_\xi$. Then the set
$\{v_{q_\xi}\}$ forms a complete orthonormal system of the space
$H$. Let $v$ be an arbitrary vector of the space $H$ and $\mu_\xi$
be a coefficient of Fourier of the vector $v$, corresponding to
$v_{q_\xi}$, in relative to the complete orthonormal system
$\{v_{q_\xi}\}$. Then, since $\sum_\xi
\mu_\xi\bar{\mu}_\xi<\infty$, then the quantity of all nonzero
elements of the set $\{\mu_\xi\}_\xi$ is not greater then the
countable cardinal number and the sequence $(\mu_n)$ of all these
nonzero numbers converges to zero.

Let $\nu_\xi$ be the $\xi$-th  coefficient of Fourier
(corresponding to $v_{q_\xi}$) of the vector $a(v)\in H$ in
relative to the complete orthonormal system $\{v_{q_\xi}\}$. Then
$\nu_\xi=\sum_\eta \lambda_{\xi\eta}\mu_\eta$ and the scalar
product $<a(v),v>$ is equal to the sum $\sum_\xi \nu_\xi\mu_\xi$.
Since the element $a(v)$ belongs to  $H$ we have the quantity of
all nonzero elements in the set $\{\nu_\xi\}_\xi$ is not greater
then the countable cardinal number and the sequence $(\nu_n)$ of
all these nonzero numbers converges to zero.

Let $\varepsilon$ be an arbitrary positive number. Then, since
quantity of nonzero numbers of the sets $\{\mu_\xi\}_\xi$ and
$\{\nu_\xi\}_\xi$ is not greater then the countable cardinal
number, and $\sum_\xi \nu_\xi\bar{\nu}_\xi<\infty$, $\sum_\xi
\mu_\xi\bar{\mu}_\xi<\infty$, then there exists
$\{f_k\}_{k=1}^l\subset \{p_i\}$ such, that for the set of indexes
$\Omega_1=\{\xi: \exists p\in \{f_k\}_{k=1}^l, q_\xi\leq p\}$ we
have
$$
\vert\sum_\xi \nu_\xi\mu_\xi-\sum_{\xi\in \Omega_1}
\nu_\xi\mu_\xi\vert<\varepsilon.
$$
Then, since quantity of nonzero numbers of the sets
$\{\mu_\xi\}_\xi$ and $\{\lambda_{\xi\eta}\}_\eta$ is not greater
then the countable cardinal number, and $\sum_\eta
\lambda_{\xi\eta}\bar{\lambda}_{\xi\eta}<\infty$, $\sum_\xi
\mu_\xi\bar{\mu}_\xi<\infty$, then there exists
$\{e_k\}_{k=1}^m\subset \{p_i\}$ such, that for the set of indexes
$\Omega_2=\{\xi: \exists p\in \{e_k\}_{k=1}^m, q_\xi\leq p\}$ we
have
$$
\vert\sum_\eta \lambda_{\xi\eta}\mu_\eta-\sum_{\eta\in \Omega_2}
\lambda_{\xi\eta}\mu_\eta\vert<\varepsilon.
$$
Hence foe the finite set $\{p_k\}_{k=1}^n=\{f_k\}_{k=1}^l\cup
\{e_k\}_{k=1}^m$ and the set $\Omega=\{\xi: \exists p\in
\{p_k\}_{k=1}^n, q_\xi\leq p\}$ of indexes we have
$$
\vert\sum_\xi \nu_\xi\mu_\xi-\sum_{\xi\in \Omega} (\sum_{\eta\in
\Omega} \lambda_{\xi\eta}\mu_\eta)\mu_\xi\vert<\varepsilon.
$$
At the same time, $<(\sum_{kl=1}^n p_kap_l)(v),v>=\sum_{\xi\in
\Omega} (\sum_{\eta\in \Omega} \lambda_{\xi\eta}\mu_\eta)\mu_\xi$.
Therefore,
$$
\vert <a(v),v>-<(\sum_{kl=1}^n p_kap_l)(v),v>\vert<\varepsilon.
$$
Hence, since the vector $v$ and the number $\varepsilon$ are
chosen arbitrarily, we have the net $(b_n^\alpha)$ ultraweakly
converges to the element $a$.

We have there exists a maximal orthogonal set $\{e_\xi\}$ of
minimal projections of the algebra $B(H)$ of all bounded linear
operators on $H$, such that the element $a$ and the set
$\{e_\xi\}$ belong to some maximal commutative subalgebra $A_o$ of
the algebra $B(H)$. We have for any finite subset
$\{p_k\}_{k=1}^n\subset \{p_i\}$ and $e\in \{e_\xi\}$ the
inequality $ e(\sum_{kl=1}^n p_kap_l)e\geq 0$ holds by the
positivity of the operator $T: b\to ebe, b\in A$.

By the previous part of the proof the net $(e_\xi b_n^\alpha
e_\xi)_{\alpha n}$ ultraweakly converges to the element $e_\xi
ae_\xi$ for any index $\xi$. Then we have $e_\xi b_n^\alpha
e_\xi\geq 0$ for all $n$ and $\alpha$. Therefore, the ultraweak
limit $e_\xi a e_\xi$ of the net $(e_\xi b_n^\alpha e_\xi)_{\alpha
n}$ is a nonnegative element. Hence, $e_\xi a e_\xi\geq 0$.
Therefore, since $e_\xi$ is chosen arbitrarily then $a\geq 0$.
\end{proof}

\begin{lem}
Let $A$ be a C$^*$-algebra of bounded linear operators on a
Hilbert space $H$, $\{p_i\}$ be an infinite orthogonal set of
projections in $A$ with the least upper bound $1$ in the algebra
$B(H)$ è $a\in A$. Then
$$
\Vert a\Vert=\sup \{\Vert \sum_{kl=1}^n p_k a p_l \Vert :n\in N,
\{p_k\}_{k=1}^n \subseteq \{p_i\}\}.
$$
\end{lem}

\begin{proof}
The inequality $-\Vert a\Vert 1\leq a \leq \Vert a\Vert 1$ holds.
Then $-\Vert a\Vert p\leq pap \leq \Vert a\Vert p$ for all natural
number $n$ and finite subset $\{p_k^\alpha\}_{k=1}^n\subset
\{p_i\}$, where $p=\sum_{k=1}^n p_k$. Therefore
$$
\Vert a\Vert\geq \sup \{\Vert \sum_{kl=1}^n p_k a p_l \Vert :n\in
N, \{p_k\}_{k=1}^n \subseteq \{p_i\}\}.
$$
At the same time, since the finite subset $\{p_k\}_{k=1}^n$ of
$\{p_i\}$ is chosen arbitrarily and by lemma 6 we have
$$
\Vert a\Vert= \sup \{\Vert \sum_{kl=1}^n p_k a p_l \Vert :n\in N,
\{p_k\}_{k=1}^n \subseteq \{p_i\}\}.
$$
Otherwise, if
$$
\Vert a\Vert>\lambda= \sup \{\Vert \sum_{kl=1}^n p_k a p_l \Vert
:n\in N, \{p_k\}_{k=1}^n \subseteq \{p_i\}\},
$$
then by lemma 3 $-\lambda 1\leq a \leq \lambda 1$. But the last
inequality is a contradiction.
\end{proof}

\begin{lem}
Let $A$ be a C$^*$-algebra of bounded linear operators on a
Hilbert space $H$, $\{p_i\}$ be an infinite orthogonal set of
projections in $A$ with the least upper bound $1$ in the algebra
$B(H)$, and let $\mathcal{A}=\{\{p_iap_j\}: a\in A\}$. Then,

1) the vector space $\mathcal{A}$ is a unit order space respect to
the order $\{p_iap_j\}\geq 0$ ($\{p_iap_j\}\geq 0$ if for any
finite subset $\{p_k\}_{k=1}^n\subset \{p_i\}$ the inequality
$pap\geq 0$ holds, where $p=\sum_{k=1}^n p_k$, and the norm
$$
\Vert \{p_iap_j\}\Vert=\sup \{\Vert \sum_{kl=1}^n p_k a p_l \Vert
:n\in N, \{p_k\}_{k=1}^n \subseteq \{p_i\}\}.
$$

2) the algebra $A$ and the unit order space $\mathcal{A}$ can be
identified as unit order spaces in the sense of the map
$$
\mathcal{I}: a\in A\to \{p_iap_j\}\in \mathcal{A}.
$$
\end{lem}

\begin{proof}
This lemma follows by lemmas 1, 3 and 4.
\end{proof}

{\it Remark.} Observe that by lemma 4 the order and the norm in
the unit order space $\mathcal{A}=\{\{p_iap_j\}: a\in A\}$ can be
defined as follows to: $\{p_iap_j\}\geq 0$, if $a\geq 0$; $\Vert
\{p_iap_j\}\Vert=\Vert a\Vert$. By lemmas 3 and 4 they are
equivalent to the order and the norm, defined in lemma 5,
correspondingly.

Let $A$ be a C$^*$-algebra, $\{p_i\}$ be a countable orthogonal
set of equivalent projections in $A$ such that $\sup_i p_i=1$ and
$$
\sum_{ij}^o p_iAp_j=\{\{a_{ij}\}:\,\text{for any indexes}\,i,j,\,
a_{ij}\in p_i Ap_j,\,\text{and}
$$
$$
\Vert \sum_{k=1,\dots,i-1}(a_{ki}+a_{ik})+a_{ii}\Vert\to
0\,\text{at}\,\, i\to \infty\}.
$$
If we introduce the componentwise algebraic operations in this
set, then $\sum_{ij}^o p_iAp_j$ becomes a vector space. Also, note
that $\sum_{ij}^o p_iAp_j$ is a vector subspace of $\mathcal{A}$.
Observe that $\sum_{ij}^o p_iAp_j$ is a normed subspace of the
algebra $\mathcal{A}$ and $\Vert \sum_{i,j=1}^n
a_{ij}-\sum_{i,j=1}^{n+1} a_{ij}\Vert\to 0$ at $n\to\infty$ for
any $\{a_{ij}\}\in \sum_{ij}^o p_iAp_j$.

Let $\sum_{ij}^o a_{ij}:=\lim_{n\to\infty} \sum_{i,j=1}^n a_{ij}$,
for any $\{a_{ij}\}\in \sum_{ij}^o p_iAp_j$, and
$$
C^*(\{p_iAp_j\}_{ij}):=\{\sum_{ij}^o a_{ij}:\{a_{ij}\}\in
\sum_{ij}^o p_iAp_j\}.
$$
Then $C^*(\{p_iAp_j\}_{ij})\subseteq A$. By lemma 5 $A$ and
$\mathcal{A}$ can be identified. We observe that, the normed
spaces $\sum_{ij}^o p_iAp_j$ and $C^*(\{p_iAp_j\}_{ij})$ can also
be identified. Further, without loss of generality we will use
these identifications.

\begin{thm}
Let $A$ be a unital C$^*$-algebra, $\{p_i\}$ be a countable
orthogonal set of equivalent projections in $A$ and $\sup_i
p_i=1$. Then $\sum_{ij}^o p_iAp_j$ is a C$^*$-subalgebra of $A$
with the componentwise algebraic operations, the associative
multiplication and the norm.
\end{thm}

\begin{proof}
We have $\sum_{ij}^o p_iAp_j$ is a normed subspace of the algebra
$A$.

Let $(a_n)$ be a sequence of elements in $\sum_{ij}^o p_iAp_j$
such that $(a_n)$ norm converges to some element $a\in A$. We have
$p_ia_np_j\to p_ia p_j$ at $n\to \infty$ for all $i$ and $j$.
Hence $p_iap_j\in p_iAp_j$ for all $i$, $j$. Let
$b^n=\sum_{k=1}^n(p_{n-1}ap_k + p_kap_{n-1})+p_nap_n$ and
$c_m^n=\sum_{k=1}^n(p_{n-1}a_mp_k + p_ka_mp_{n-1})+p_na_mp_n$, for
any $n$. Then $c_m^n\to b^n$ at $m\to \infty$. It should be proven
that $\Vert b_n\Vert\to 0$ at $n\to\infty$.

Let $\varepsilon \in {\Bbb R}_+$. Then there exists $m_o$ such
that $\Vert a-a_m\Vert<\varepsilon$ for any $m>m_o$. Also for all
$n$ and $\{p_k\}_{k=1}^n\subset \{p_i\}$ $\Vert
(\sum_{k=1}^np_k)(a-a_m)(\sum_{k=1}^np_k)\Vert<\varepsilon$. Hence
$\Vert b^n-c_m^n\Vert<2\varepsilon$ for any $m>m_o$. At the same
time, $\Vert b^n-c_{m_1}^n\Vert<2\varepsilon$, $\Vert
b^n-c_{m_2}^n\Vert<2\varepsilon$ for all $m_o<m_1$, $m_2$. Since
$(a_n)\subset \sum_{ij}^o p_iAp_j$ then for any $m$ $\Vert
c_m^n\Vert\to 0$ at $n\to \infty$. Hence, since $\Vert
c_{m_1}^n\Vert\to 0$ and $\Vert c_{m_2}^n\Vert\to 0$ at $n\to
\infty$, then there exists $n_o$ such that $\Vert
c_{m_1}^n\Vert<\varepsilon$, $\Vert c_{m_2}^n\Vert<\varepsilon$
and $\Vert c_{m_1}^n+c_{m_2}^n\Vert<2\varepsilon$ for any $n>n_o$.
Then $\Vert 2b_n\Vert=\Vert
b^n-c_{m_1}^n+c_{m_1}^n+c_{m_2}^n+b^n-c_{m_2}^n\Vert\leq \Vert
b^n-c_{m_1}^n\Vert+\Vert c_{m_1}^n+c_{m_2}^n\Vert+\Vert
b^n-c_{m_2}^n\Vert<2\varepsilon+2\varepsilon+2\varepsilon=6\varepsilon$
for any $n>n_o$, i.e. $\Vert b_n\Vert<3\varepsilon$ for any
$n>n_o$. Since $\varepsilon$ is arbitrarily chosen then $\Vert
b_n\Vert\to 0$ at $n\to\infty$. Therefore $a\in\sum_{ij}^o
p_iAp_j$. Since the sequence $(a_n)$ is arbitrarily chosen then
$\sum_{ij}^o p_iAp_j$ is a Banach space.

Let $\{a_{ij}\}$, $\{b_{ij}\}$ be arbitrary elements of the Banach
space $\sum_{ij}^o p_iAp_j$. Let $a_m=\sum_{kl=1}^m a_{kl}$,
$b_m=\sum_{kl=1}^m b_{kl}$ for all natural numbers $m$. We have
the sequence $(a_m)$ converges to $\{a_{ij}\}$ and the sequence
$(b_m)$ converges to $\{b_{ij}\}$ in $\sum_{ij}^o p_iAp_j$. Also
for all $n$ and $m$ $a_mb_n\in \sum_{ij}^o p_iAp_j$. Then for any
$n$ the sequence $(a_mb_n)$ converges to $\{a_{ij}\}b_n$ at $m\to
\infty$. Hence $\{a_{ij}\}b_n\in \sum_{ij}^o p_iAp_j$. Note that
$\sum_{ij}^o p_iAp_j\subseteq A$. Therefore for any
$\varepsilon\in {\Bbb R}_+$ there exists $n_o$ such that
$\Vert\{a_{ij}\}b_{n+1}-\{a_{ij}\}b_n\Vert\leq
\Vert\{a_{ij}\}\Vert \Vert b_{n+1}-b_n\Vert\leq \varepsilon$ for
any $n>n_o$. Hence the sequence $(\{a_{ij}\}b_n)$ converges to
$\{a_{ij}\}\{b_{ij}\}$ at $n\to \infty$. Since $\sum_{ij}^o
p_iAp_j$ is a Banach space then $\{a_{ij}\}\{b_{ij}\}\in
\sum_{ij}^o p_iAp_j$. Since $\sum_{ij}^o p_iAp_j\subseteq A$, then
$\sum_{ij}^o p_iAp_j$ is a C$^*$-algebra.
\end{proof}

Let $H$ be an infinite dimensional Hilbert space, $B(H)$ be the
algebra of all bounded linear operators. Let $\{p_i\}$ be a
countable orthogonal set of equivalent projections in $B(H)$ and
$\sup_i p_i=1$. Let $\{\{p_j^i\}_j\}_i$ be the set of infinite
subsets of $\{p_i\}$ such that for all distinct $\xi$ and $\eta$
$\{p_j^\xi\}_j\cap \{p_j^\eta\}_j=\oslash$,
$\vert\{p_j^\xi\}_j\vert=\vert \{p_j^\eta\}_j\vert$  and
$\{p_i\}=\cup_i \{p_j^i\}_j$. Then let $q_i=\sup_j p_j^i$ in
$B(H)$, for all $i$. Then $\sup_i q_i=1$ and $\{q_i\}$ be a
countable orthogonal set of equivalent projections. Then we say
that the countable orthogonal set $\{q_i\}$ of equivalent
projections {\it is defined by the set} $\{p_i\}$ in $B(H)$. We
have the next corollary.

\begin{cor}
Let $A$ be a unital C$^*$-algebra of bounded linear operators in a
Hilbert space $H$, $\{p_i\}$ be a countable orthogonal set of
equivalent projections in $A$ and $\sup_i p_i=1$. Let $\{q_i\}$ be
a countable orthogonal set of equivalent projections in $B(H)$
defined by the set $\{p_i\}$ in $B(H)$. Then $\sum_{ij}^o q_iAq_j$
is a C$^*$-subalgebra of the algebra $A$.
\end{cor}

\begin{proof}
Let $\{\{p_j^i\}_j\}_i$ be the set of infinite subsets of
$\{p_i\}$ such that for all distinct $\xi$ and $\eta$
$\{p_j^\xi\}_j\cap \{p_j^\eta\}_j=\oslash$,
$\vert\{p_j^\xi\}_j\vert=\vert \{p_j^\eta\}_j\vert$ and
$\{p_i\}=\cup_i \{p_j^i\}_j$. Then let $q_i=\sup_j p_j^i$ in
$B(H)$, for all $i$. Then we have for all $i$ and $j$
$q_iAq_j=\{\{p_\xi^iap_\eta^j\}_{\xi\eta}: a\in A\}$. Hence
$q_iAq_j\subset A$ for all $i$ and $j$.

The rest part of the proof is the repeating of the proof of
theorem 6.
\end{proof}

{\it Example.} {\it 1.} Let $\mathcal{M}$ be the closure on the
norm of the inductive limit $\mathcal{M}_o$ of the inductive
system
$$
C\to M_2(C)\to M_3(C)\to M_4(C)\to \dots,
$$
where $M_n(C)$ is mapped into the upper left corner of
$M_{n+1}(C)$. Then $\mathcal{M}$ is a C$^*$-algebra (\cite{Bra}).
The algebra $\mathcal{M}$ contains the minimal projections of the
form $e_{ii}$, where $e_{ij}$ is an infinite dimensional matrix,
whose $(i,i)$-th component is $1$ and the rest components are
zeros. These projections form the countable orthogonal set
$\{e_{ii}\}_{i=1}^\infty$ of minimal projections. Let
$$
M_n^o({\Bbb C})=\{\sum_{ij}\lambda_{ij}e_{ij}:\,\text{for any
indexes}\,i,j, \lambda_{ij}\in {\Bbb C},\,\text{and}
$$
$$\Vert
\sum_{k=1,\dots,i-1}(\lambda_{ki}e_{ki}+\lambda_{ik}e_{ik})+\lambda_{ii}e_{ii}\Vert\to
0\,\text{at}\,\, i\to \infty\}.
$$
Then ${\Bbb C}\cdot 1+M_n^o({\Bbb C})=\mathcal{M}$ (see
\cite{arz}) and by theorem 6 $M_n^o({\Bbb C})$ is a simple
C$^*$-algebra. Note that there exists a mistake in the formulation
of theorem 3 in \cite{arz}. ${\Bbb C}\cdot 1+M_n^o({\Bbb C})$ is a
C$^*$-algebra. But the algebra ${\Bbb C}\cdot 1+M_n^o({\Bbb C})$
is not simple. Because ${\Bbb C}\cdot 1+M_n^o({\Bbb C})\neq
M_n^o({\Bbb C})$ and $M_n^o({\Bbb C})$ is an ideal of the algebra
${\Bbb C}\cdot 1+M_n^o({\Bbb C})$, i.e. $[{\Bbb C}\cdot
1+M_n^o({\Bbb C})]\cdot M_n^o({\Bbb C})\subseteq M_n^o({\Bbb C})$.

{\it 2.} There exist a C$^*$-algebra $A$ and different countable
orthogonal sets $\{e_i\}$ and $\{f_i\}$ of equivalent projections
in $A$ such that $\sup_i e_i=1$, $\sup_i f_i=1$, $\sum_{ij}^o
e_iAe_j\neq \sum_{ij}^o f_iAf_j$. Indeed, let $H$ be an infinite
dimensional Hilbert space, $B(H)$ be the algebra of all bounded
linear operators. Let $\{p_i\}$ be a countable orthogonal set of
equivalent projections in $B(H)$ and $\sup_i p_i=1$. Then
$\sum^o_{ij} p_iB(H)p_j\subset B(H)$. Let $\{\{p_j^i\}_j\}_i$ be
the set of infinite subsets of $\{p_i\}$ such that for all
distinct $\xi$ and $\eta$ $\{p_j^\xi\}_j\cap
\{p_j^\eta\}_j=\oslash$, $\vert\{p_j^\xi\}_j\vert=\vert
\{p_j^\eta\}_j\vert$ and $\{p_i\}=\cup_i \{p_j^i\}_j$. Then let
$q_i=\sup_j p_j^i$ for all $i$. Then $\sup_i q_i=1$ and $\{q_i\}$
be a countable orthogonal set of equivalent projections. We assert
that $\sum^o_{ij} p_iB(H)p_j\neq \sum^o_{ij} q_iB(H)q_j$. Indeed,
let $\{x_{ij}\}$ be a set of matrix units constructed by the
infinite set $\{p_j^1\}_j\in \{\{p_j^i\}_j\}_i$, i.e. for all $i$,
$j$, $x_{ij}x_{ij}^*=p_i^1$, $x_{ij}^*x_{ij}=p_j^1$,
$x_{ii}=p_i^1$. Then the von Neumann algebra $\mathcal{N}$
generated by the set $\{x_{ij}\}$ is isometrically isomorphic to
$B(\mathcal{H})$ for some Hilbert space $\mathcal{H}$. We note
that $\mathcal{N}$ is not subset of $\sum^o_{ij} p_iB(H)p_j$. At
the same time, $\mathcal{N}\subseteq \sum^o_{ij} q_iB(H)q_j$ and
$\sum^o_{ij} p_i^1\mathcal{N}p_j^1\subseteq \sum^o_{ij}
p_iB(H)p_j$.

\begin{thm}
Let $A$ be a unital simple C$^*$-algebra of bounded linear
operators in a Hilbert space $H$, $\{p_i\}$ be a countable
orthogonal set of equivalent projections in $A$ and $\sup_i
p_i=1$. Let $\{q_i\}$ be a countable orthogonal set  of equivalent
projections in $B(H)$ defined by the set $\{p_i\}$ in $B(H)$. Then
$\sum_{ij}^o q_iAq_j$ is a simple C$^*$-algebra.
\end{thm}

\begin{proof} By theorem 6 $\sum_{ij}^o p_iAp_j$ is a C$^*$-algebra.
Let $\{\{p_j^i\}_j\}_i$ be the set of infinite subsets of
$\{p_i\}$ such that for all distinct $\xi$ and $\eta$
$\{p_j^\xi\}_j\cap \{p_j^\eta\}_j=\oslash$,
$\vert\{p_j^\xi\}_j\vert=\vert \{p_j^\eta\}_j\vert$ and
$\{p_i\}=\cup_i \{p_j^i\}_j$. Then let $q_i=\sup_j p_j^i$ in
$B(H)$, for all $i$. Then we have for all $i$ and $j$
$q_iAq_j=\{\{p_\xi^iap_\xi^j\}: a\in A\}$. Hence $q_iAq_j\subset
A$ for all $i$ and $j$. By corollary 7 $\sum_{ij}^o q_iAq_j$ is a
C$^*$-algebra.

Since projections of the set $\{p_i\}$ pairwise equivalent then
the projection $q_i$ is equivalent to $1\in A$ for any $i$. Hence
$q_iAq_i\cong A$ and $q_iAq_i$ is a simple C$^*$-algebra for any
$i$.

Let $q$ be an arbitrary projection in $\{q_i\}$. Then $qAq$ is a
subalgebra of $\sum_{ij}^o q_iAq_j$. Let $I$ be a closed ideal of
the algebra $\sum_{ij}^o q_iAq_j$. Then $IqAq\subset I$ and
$Iq\cdot qAq\subset Iq$. Therefore $qIqqAq\subseteq qIq$, that is
$qIq$ is a closed ideal of the subalgebra $qAq$. Since $qAq$ is
simple then $qIq=qAq$.

Let $q_1$, $q_2$ be arbitrary projections in $\{q_i\}$. We assert
that $q_1Iq_2=q_1Aq_2$ and $q_2Iq_1=q_2Aq_1$. Indeed, we have the
projection $q_1+q_2$ is equivalent to $1\in A$. Let $e=q_1+q_2$.
Then $eAe\cong A$ and $eAe$ is a simple C$^*$-algebra. At the same
time we have $eAe$ is a subalgebra of $\sum_{ij}^o q_iAq_j$ and
$I$ is an ideal of $\sum_{ij}^o q_iAq_j$. Hence $IeAe\subset I$
and $Ie\cdot eAe\subset Ie$. Therefore $eIeeAe\subseteq eIe$, that
is $eIe$ is a closed ideal of the subalgebra $eAe$. Since $eAe$ is
simple then $eIe=eAe$. Hence $q_1Iq_2=q_1Aq_2$ and
$q_2Iq_1=q_2Aq_1$. Therefore $q_iIq_j=q_iAq_j$ for all $i$ and
$j$. We have $I$ is norm closed. Hence $I=\sum_{ij}^o q_iAq_j$,
i.e. $\sum_{ij}^o q_iAq_j$ is a simple C$^*$-algebra.
\end{proof}

\section{Applications}

\begin{thm}
Let $\mathcal{N}$ be a W$^*$-factor of type II$_\infty$ of bounded
linear operators in a Hilbert space $H$, $\{p_i\}$ be a countable
orthogonal set of equivalent projections in $\mathcal{N}$ and
$\sup_i p_i=1$. Then for any countable orthogonal set $\{q_i\}$ of
equivalent projections in $B(H)$ defined by the set $\{p_i\}$ in
$B(H)$ the C$^*$-algebra $\sum_{ij}^o q_i\mathcal{N}q_j$ is a
C$^*$-factor with a nonzero finite and an infinite projection. In
this case $\sum_{ij}^o q_i\mathcal{N}q_j$ is not a von Neumann
algebra.
\end{thm}

\begin{proof}
By the definition of the set $\{q_i\}$ we have $\sup_i q_i=1$ and
$\{q_i\}$ be a countable orthogonal set of equivalent {\it
infinite} projections. By theorem 6 we have $\sum_{ij}^o
q_i\mathcal{N}p_j$ is a C$^*$-subalgebra of $\mathcal{N}$. Let $q$
be a nonzero finite projection of $\mathcal{N}$. Then there exists
a projection $p\in \{q_i\}$ such that $qp\neq 0$. We have
$q\mathcal{N} q$ is a finite von Neumann algebra. Let $x=pq$. Then
$x\mathcal{N}x^*$ is a weakly closed C$^*$-subalgebra. Note that
the algebra $x\mathcal{N}x^*$ has a center-valued faithful trace.
Let $e$ be a nonzero projection of the algebra $x\mathcal{N}x^*$.
Then $ep=e$ and $e\in p\mathcal{N}p$. Hence $e\in \sum_{ij}^o
q_i\mathcal{N}q_j$. We have the weak closure of $\sum_{ij}^o
q_i\mathcal{N}q_j$ in the algebra $\mathcal{N}$ coincides with
this algebra $\mathcal{N}$. Then by the weak continuity of the
multiplication $\sum_{ij}^o q_i\mathcal{N}q_j$ is a factor. Note
since $1\notin \sum_{ij}^o q_i\mathcal{N}q_j$ then $\sum_{ij}^o
q_i\mathcal{N}q_j$ is not weakly closed in $\mathcal{N}$. Hence
the C$^*$-factor $\sum_{ij}^o q_i\mathcal{N}q_j$ is not a von
Neumann algebra.
\end{proof}

{\it Remark.} Note that, in the article \cite{Ror3} a simple
C$^*$-algebra with an infinite and a nonzero finite projection
have been constructed by M.R{\o}rdam. In the next corollary we
construct a simple purely infinite C$^*$-algebra. Note that simple
purely infinite C$^*$-algebras are considered and investigated, in
particular, in \cite{Ph} and \cite{Ror}.

\begin{thm}
Let $\mathcal{N}$ be a W$^*$-factor of type III of bounded linear
operators in a Hilbert space $H$. Then for any countable
orthogonal set $\{p_i\}$ of equivalent projections in
$\mathcal{N}$ such that $\sup_i p_i=1$, $\sum_{ij}^o
p_i\mathcal{N}p_j$ is a simple purely infinite C$^*$-algebra. In
this case $\sum_{ij}^o p_i\mathcal{N}p_j$ is not a von Neumann
algebra.
\end{thm}

\begin{proof}
Let $p_{i_o}$ be a projection in $\{p_i\}$. We have the projection
$p_{i_o}$ can be exhibited as a least upper bound of a countable
orthogonal set $\{p_{i_o}^j\}_j$ of equivalent projections in
$\mathcal{N}$. Then for any $i$ the projection $p_i$ has a
countable orthogonal set $\{p_i^j\}_j$ of equivalent projections
in $\mathcal{N}$ such that the set $\bigcup_i \{p_i^j\}_j$ is a
countable orthogonal set of equivalent projections in
$\mathcal{N}$. Hence the set $\{p_i\}$ is defined by the set
$\bigcup_i \{p_i^j\}_j$ in $B(H)$ (in $\mathcal{N}$). Hence by
theorem 8 $\sum_{ij}^o p_i\mathcal{N}p_j$ is a simple
C$^*$-algebra. Note since $1\notin \sum_{ij}^o p_i\mathcal{N}p_j$
then $\sum_{ij}^o p_i\mathcal{N}p_j$ is not weakly closed in
$\mathcal{N}$. Hence $\sum_{ij}^o p_i\mathcal{N}p_j$ is not a von
Neumann algebra.

Suppose there exists a nonzero finite projection $q$ in
$\sum_{ij}^o p_i\mathcal{N}p_j$. Then there exists a projection
$p\in \{p_i\}$ such that $qp\neq 0$. We have $q(\sum_{ij}^o
p_i\mathcal{N}p_j)q$ is a finite C$^*$-algebra. Let $x=pq$. Then
$x\mathcal{N}x^*$ is a C$^*$-subalgebra. Moreover
$x\mathcal{N}x^*$ is weakly closed and $x\mathcal{N}x^*\subset
p\mathcal{N}p$. Hence $x\mathcal{N}x^*$ has a center-valued
faithful trace. Then $x\mathcal{N}x^*$ is a finite von Neumann
algebra with a center-valued faithful normal trace. Let $e$ be a
nonzero projection of the algebra $x\mathcal{N}x^*$. Then $ep=e$
and $e\in p\mathcal{N}p$. Hence $e\in \mathcal{N}$. This is a
contradiction.
\end{proof}

{\it Example.} Let $H$ be a separable Hilbert space and $B(H)$ the
algebra of all bounded linear operators on $H$. Let $\{q_i\}$ be a
maximal orthogonal set of equivalent minimal projections in
$B(H)$. Then $\sum_{ij}^o q_i B(H)q_j$ is a two sided closed ideal
of the algebra $B(H)$. Using the set $\{q_i\}$ we construct a
countable orthogonal set $\{p_i\}$ of equivalent infinite
projections such that $\sup_i p_i=1$. Let $\{\{q_j^i\}_j\}_i$ be
the countable set of countable subsets of $\{q_i\}$ such that for
all distinct $i_1$ and $i_2$ $\{q_j^{i_1}\}_j\cap
\{p_j^{i_2}\}_j=\oslash$ and $\{q_i\}=\cup_i \{q_j^i\}_j$. Then
let $p_i=\sup_j q_j^i$ for all $i$. Then $\sup_i p_i=1$ and
$\{p_i\}$ is a countable orthogonal set of equivalent infinite
projections in $B(H)$ defined by $\{q_i\}$ in $B(H)$.

Let $\{q_{nm}^{ij}\}$ be the set of matrix units constructed by
the set $\{\{q_j^i\}_j\}_i$, i.e.
$q_{nm}^{ij}{q_{nm}^{ij}}^*=q_n^i$,
${q_{nm}^{ij}}^*q_{nm}^{ij}=q_m^j$, $q_{nn}^{ii}=q_n^i$ for all
$i$, $j$,$n$,$m$. Then let $a=\{a_{nm}^{ij}q_{nm}^{ij}\}$ be the
decomposition of the element $a\in B(H)$, where the components
$a_{nm}^{ij}$ are defined as follows
$$
a_{11}^{11}=\lambda, a_{12}^{21}=\lambda,
a_{13}^{31}=\lambda,\dots,a_{1n}^{n1}=\lambda,\dots,
$$
and the rest components $a_{nm}^{ij}$ are zero, i.e.
$a_{nm}^{ij}=0$. Then $p_1a=a$. Then since $a\notin \sum_{ij}^o
p_i B(H)p_j$ and $p_1\in \sum_{ij}^o p_i B(H)p_j$ then
$\sum_{ij}^o p_i B(H)p_j$ is not an ideal of $B(H)$. But by
theorem 6 $\sum_{ij}^o p_i B(H)p_j$ is a C$^*$-algebra. Hence
there exists a C$^*$-algebra $A$ with an orthogonal set $\{p_i\}$
of equivalent projections such that $\sum_{ij}^o p_iAp_j$ is not
an ideal of $A$.

\end{document}